\documentclass[12pt]{article}
\usepackage{amssymb,amsthm,hyperref,amsmath,bm,amsfonts}
\usepackage{arydshln}
\usepackage{verbatim}
\usepackage{graphicx}
\usepackage{float}

\usepackage{subfigure}
\usepackage{multirow}
\usepackage{color}
\usepackage{appendix}

\usepackage{latexsym,amsmath,amsfonts,amscd}
\usepackage{epsfig}
\usepackage{changebar}
\usepackage{pstricks}
\usepackage{pst-plot}
\usepackage{multirow}
\usepackage{subfigure}
\usepackage{placeins}
\usepackage{amssymb}
\usepackage{enumerate}

\newtheorem{theorem}{Theorem}[section]
\newtheorem{lemma}[theorem]{Lemma}

\newtheorem{definition}{Definition}[section]%
\newtheorem{assumption}{Assumption}[section]
\newtheorem{remark}{Remark}[section]
\numberwithin{equation}{section}

\title{Exponential turnpike property for particle systems and mean-field limit}
\author{
Michael Herty \thanks{Chair in Numerical Analysis, IGPM, RWTH Aachen University, Templergraben, 55, D-52062 Aachen, Germany (herty@igpm.rwth-aachen.de)}\qquad
Yizhou Zhou\thanks{IGPM, RWTH Aachen University, Templergraben, 55, D-52062 Aachen, Germany (zhou@igpm.rwth-aachen.de)} 
}

\date{\today}
\begin{document}
\maketitle{}

\begin{abstract}
    This work is concerned with the exponential turnpike property for optimal control problems of  particle systems and their mean-field limit. Under the  assumption of the strict dissipativity of  the cost function,   exponential estimates for both optimal states and optimal control are proven. Moreover, we show that all the results for particle systems can be preserved under the limit in the case of infinitely many particles.
\end{abstract}

\hspace{-0.5cm}\textbf{Keywords:}
\small{Exponential turnpike property, mean-field limit, optimal control}\\

\hspace{-0.5cm}\textbf{AMS subject classification:} \small{93C20, 35L04, 35Q89, 49N10}

\section{Introduction}

For optimal control problems of time-dependent differential equations, the exponential turnpike property states, that the optimal solution remains (exponentially) close to a reference solution. Usually, this reference solution is taken as the optimal solution to the corresponding static problem. The concept of turnpike was first introduced  
for discrete-time optimal control problems \cite{DSS,Sa}. Since then, many turnpike results have been established and there has been recent interest in the mathematical community \cite{MR3217211,MR3470445,MR4079009,receding,use-2,MR4402854}.

% In many of cases, the turnpike property is quantified by an exponential estimate.

In the present work, we focus on the exponential turnpike phenomenon for optimal control problems of a class of interacting particle systems and their mean-field limit equations. Important applications for these systems occur in the fields of swarm robotics \cite{CKPP}, crowd dynamics \cite{ABCK}, traffic management \cite{TZ}, or opinion dynamics \cite{AHP} to name but a few.

The original formulation of the interacting particle system is usually at the so--called microscopic level and given by a coupled system of  ODEs. 
Alternatively, one can also focus on the collective behavior by considering the probability density distribution of the particles and investigating the corresponding McKean-Vlasov or mean-field equation, see e.g. \cite{MR3642940,MR3969953,MR4014449} for results involving control actions. The control of large--scale interacting particle systems has gained recent interest, see e.g., 
 \cite{AHP,sparse,sparse-Cucker}. The control of high--dimensional system is challenging and current approaches resort to e.g. using Riccati--based  \cite{MR3431287,MR4469721}, moment--driven control \cite{MR4399019}, or model predicitive control approaches \cite{MR3894072,AHP,MR4288153}. Motivated by this, we aim to utilize the turnpike property 
 to control those high--dimensional systems~\cite{use-1,use-1.1,MR3973342}. More precisely,  we prove the exponential turnpike estimate for ODE systems with an arbitrary particle number and show that the property also holds in the mean-field limit. Here, we utilize the particular structure of interacting particle systems to derive the turnpike property. 

The topic of turnpike property for mean-field optimal control problems has been studied recently in \cite{main}. At this point, we would compare \cite{main} with the present paper and point out our main contributions. (1) In \cite{main}, the authors prove the turnpike property with interior decay \cite{interior-original}  which is a time integral property \cite{MR4493557}. In the present paper, under similar assumptions (with a minor modification), we present a point-wise exponential estimate which is more quantitative. (2) In addition to the estimate of the optimal solution, we also prove the exponential decay for the optimal control. 

As in \cite{main}, our basic assumption is that the optimal control problems satisfy a strict dissipativity inequality. By considering a feedback control, we obtain the cheap control inequality. Then, we use this inequality iteratively to prove the exponential estimate for the optimal solution. 
This iteration technique has also been used to prove the turnpike property for other optimal control problems (See i.e \cite{EGPZ2}).
Note that all the estimates for particle systems are independent of the particle number $N$. Thus all  results are also expected in the mean-field level as $N\rightarrow \infty$. By using convergence in the Wasserstein distance and the lower semi-continuity of the cost function, we prove the corresponding exponential decay property for the solution of the mean-field optimal control problem.
In order to establish the exponential decay for the optimal control, we design a specific feedback control, see also  \cite{EGPZ}. In this way, the optimal control can be bounded by the optimal solution. Combining with the estimate for solutions, we also prove the exponential decay property for the optimal control with respect to time $t$. 
% In addition, we present some remarks to illustrate that the results for particle systems and for mean-field equations are connected in the sense of mean-field limit. 

% The basic assumptions: quadratic upper bound for the cost function, this is reasonable since the examples  

The paper is organized as follows. In Section \ref{Section2}, we state the problem and present some basic assumptions. 
% ove the exponential turnpike property for the optimal control problems of particle systems. This section consists of two parts. 
In Section \ref{Section2.1}, we prove the cheap control property for the optimal control problem of the particle system. By considering the limit $N\rightarrow \infty$, we prove the same property in the mean-field level. Based on these results, we prove the exponential turnpike property for both the particle system and the mean-field problem in Section \ref{Section 4}. At last, the auxiliary estimate in the Wasserstein distance is given in Appendix \ref{sectionA}. 
The main results are Theorem \ref{thm14} on the exponential turnpike property for the particle system and Theorem \ref{theorem-m-s}-\ref{thm14-c} for the mean-field problem.

% {\color{blue} Maybe you can add here two sentences saying, that the main result is Theorem XXX on the turnpike property on particle/mean field? }

\section{Preliminaries}\label{Section2}

% \subsection{Optimal control problem in the mean-field level}
Consider the optimal control problem $\mathcal{Q}(0,T,\mu_0)$:
\begin{align}
\mathcal{V}(0,T,\mu_0)=&~\min_{u\in\mathcal{F}}  \int_{0}^T f(\mu(t,x),u(t,x)) dt, \nonumber \\[2mm]
    :=&~\min_{u\in\mathcal{F}} \int_{0}^T \int L(x)d\mu(t,x) dt + \int_{0}^T \int \Psi(u(x,t))d\mu(t,x) dt. \label{mec}
\end{align}
Here $\mu(t,\cdot)\in P_2(\mathbb{R}^d)$ is a probability measure
on $\mathbb{R}^d$ defined for $t\in [0, T]$ and it satisfies the following equation in a distributional sense
\begin{eqnarray}\label{meq}
    \begin{aligned}
    &\partial_t\mu + \nabla_x\cdot\Big((P\ast \mu + u)\mu\Big)  =0,\qquad 0<t<T,\quad x\in\mathbb{R}^d,\\[2mm]
    &\mu(0,x) = \mu_0(x).
\end{aligned}
\end{eqnarray}
Here, $P(x)\in\mathbb{R}^d$ is a vector-valued function and
$$
(P \ast \mu)(x,t) = \int_{\mathbb{R}^d} P(x-y)d\mu(t,y).
$$
As that in \cite{FS}, we take the control $u(t,x)\in \mathcal{F}$ satisfying
\begin{definition}
    Fix a control bound $0<C_B<\infty$. Then $u(t,x)\in\mathcal{F} $ if and only if
    \begin{itemize}
        \item[(i)] $u:[0, T] \times \mathbb{R}^d \rightarrow \mathbb{R}^d$ is a Carathéodory function. 
        \item[(ii)] $u(t,\cdot)\in W^{1,\infty}_{loc}(\mathbb{R}^d)$ 
 for almost every $t \in [0, T]$.
        \item[(iii)] $|u(t,0)| + \|u(t,\cdot)\|_{Lip} \leq C_B$ for almost every $t \in [0, T]$.
    \end{itemize}
\end{definition}
\begin{remark}
    In \cite{FS}, the control bound can be chosen as an integrable function $l(t)\in L^q(0,T)$ for $1\leq q <\infty$. For simplicity, we take the bound to be constant.
\end{remark}

Next, we show assumptions for the optimal control problem \eqref{meq}. 
\begin{assumption}\label{ass1}
    The cost function $f$ satisfies the following assumptions:
    \begin{itemize}
        \item[(i)] Strict dissipativity: there exists a constant $C_D$ such that, for any $b\geq a\geq 0$ and any pairs $(\mu(t,x),u(t,x))\in P_2(\mathbb{R}^d)\times \mathcal{F}$, the following inequality holds
        $$
        \int_a^b f(\mu(t,x),u(t,x)) dt \geq C_D 
        \int_a^b \int_{\mathbb{R}^d}\Big(|x-\bar{x}|^2 +|u(t,x)|^2 \Big) d\mu(t,x)dt .
        $$
        \item[(ii)] There exist constants $C_{\Psi}$ and $C_{L}$ such that 
\begin{equation}
  \Psi(u)\leq C_{\Psi}|u|^2\quad \text{and} \quad  
L(x)\leq C_L |x-\bar{x}|^2.  
\end{equation}
It holds for all $x\in B(\bar{x}, R) := \{x \in \mathbb{R}^d : |x-\bar{x}| < R\}$ and $u\in B(0,R)$.
        \item[(iii)] 
        The interaction function $P(x)$ satisfies $P(0)=0$ and the following Lipschitz property:
\begin{equation}\label{Lip-P}
|P(x)-P(y)|\leq C_p|x-y|,\qquad \forall ~x\in \mathbb{R}^d
\end{equation}
with $C_P>0$ a constant.

    \end{itemize}
\end{assumption}

% Here and in what follows, for vectors $a=(a_1,a_2,...,a_d)\in \mathbb{R}^d$ and $b=(b_1,b_2,...,b_d) \in \mathbb{R}^d$ we denote 
% $$
% \langle a, b\rangle = a\cdot b = \sum_{i=1}^da_ib_i,\qquad 
% |a| = \sqrt{\langle a, a\rangle}.
% $$ 
\begin{remark}
    These assumptions are also used in \cite{main} except for condition (ii). Here, we need to assume that both $\Psi$ and $L$ can be bounded by quadratic functions. Note that, this assumption is also satisfied for the example of \cite{main}. 
\end{remark}

For further discussion of the optimal control problem, we consider the empirical measure on $[0,T] \times \mathbb{R}^d$
\begin{equation}\label{empirical}
\mu_{N}(t,x) = \frac{1}{N}\sum_{i=1}^N \delta\left(x-x_i(t)\right).
\end{equation}
Here, $x_i(t)$ $(i=1,2,...,N)$ is the solution to the optimal control problem $\mathcal{Q}_N(0,T,x_0)$:
\begin{eqnarray}\label{particle-pb}
\begin{aligned}
    \mathcal{V}_N(0,T,x_0)
    % =&~\min_{u}  \int_{0}^T f_N(x_i(t),u(t)) dt, \\[2mm]
    =&~\min_{u_N\in\mathcal{F}}  \frac{1}{N}\sum_{i=1}^N\int_{0}^T L(x_i(t)) + \Psi(u_N(t,x_i(t))) dt, \\[2mm]
    \frac{dx_i(t)}{dt} =&~ \frac{1}{N}\sum_{j=1}^NP(x_i(t)-x_j(t)) + u_N(t,x_i(t)), \\[2mm]
    x_i(0) =&~ x_{i0}.
\end{aligned}
\end{eqnarray}
Here, $x(t)=(x_1(t),x_2(t),...,x_N(t))$ represents $N$ particles, $x_0=(x_{10},x_{20},...,x_{N0})$ is the initial data, and $u_{N}(t,x_i(t))$ is the control. We use the subscript $N$ to emphasize the dependence of the optimal control $u_N$ of \eqref{particle-pb}  on the number of particles $N$.

\begin{remark}\label{remark21}
Problem $\mathcal{Q}_N(0,T,x_0)$ can be formally derived from the original optimal control problem. For any $N$ we have 
\begin{align*}
    f(\mu_{N},u_N) =&~ \int L(x)d\mu_{N}(t,x)  +  \int \Psi(u_N(t,x))d\mu_{N}(t,x)\\[2mm] =&~ \frac{1}{N}\sum_{i=1}^N \Big[L(x_i(t)) + \Psi(u_N(t,x_i(t)))\Big].
\end{align*}
which implies that the cost function in \eqref{particle-pb} is given by 
$$
\mathcal{V}_N(0,T,x_0)
    =\min_{u_N\in\mathcal{F}} \int_{0}^T f(\mu_{N},u_N) dt.
$$
\end{remark}

%% tbc

As outlined in the remark, the optimal control problem \eqref{particle-pb} and the original problem are intertwined. Under Assumptions \ref{ass1}, the existence and uniqueness of the problem \eqref{meq}-\eqref{mec} has been established in \cite{FS}. To recall the theorem,  the definition of the $p-$Wasserstein distance between two probability measures $\mu$ and $\nu$ is given:
$$
\mathcal{W}_p(\mu,\nu) = \inf_{\gamma \in \Gamma(\mu,\nu)} \left( \int_{\mathbb{R}^{2d}} |x-y|^p d\gamma(x,y) \right)^{1/p}.
$$
Here, $\Gamma(\mu,\nu)$ denotes the set of transport plans, i.e., collection of all probability measures with marginals $\mu$ and $\nu$, see also \cite{MR2459454}.

\begin{theorem}\label{thm21}
Assume that the initial data $\mu_0$ in \eqref{meq} is compactly supported, i.e., there exists $R > 0$ such that
$\text{supp}~ \mu_0 \subset B(0, R)$. Moreover, the empirical measure
$\mu_N(x,0)$ converges to $\mu_0$ in $\mathcal{W}_1$ distance. Then, there exists an optimal control $u(t,x)$ and a weak equi-compactly supported solution $\mu(t,x)$ to the problem \eqref{mec}-\eqref{meq}. Namely, for all $t\in [0,T]$ the distribution $\mu(t,x) \in C([0, T];P_1(\mathbb{R}^d))$ satisfies $\text{supp}~ \mu(t,\cdot) \subset B(0, R)$ and 
\begin{eqnarray}\label{weak}
\begin{aligned}
    &\int \phi(t,x)d\mu(t,x) -\int \phi(0,x)d\mu_0(x) \\[2mm]
    = &\int_{0}^{t} \int \Big[\partial_t \phi + \nabla_x \phi \cdot (P\ast \mu + u) \Big]d\mu(s,x)ds,\qquad \forall~\phi\in C_0^{\infty}([0,T]\times \mathbb{R}^d). 
\end{aligned}
\end{eqnarray}
The optimal solution satisfies
$$
\lim_{k\rightarrow \infty}\mathcal{W}_1(\mu_{N_k}(t,\cdot), \mu(t)) = 0
$$
uniformly with respect to $t \in [0, T]$ and $u_{N_k}$ converges to $u$ in $\mathcal{F}$. Here $\mu_{N_k}$ is given by \eqref{empirical} and $(x_i(t),u_{N_k}(t,x))$ is the optimal solution to \eqref{particle-pb} with $N_k$ particles. Moreover, $f$ is lower semi-continuous:
\begin{equation}
\int_0^T f(\mu(t,x),u(t,x))dt\leq \liminf_{k\rightarrow \infty}\int_0^T f(\mu_{N_k}(t,x),u_{N_k}(t,x))dt.    
\end{equation}
\end{theorem}

% In addition, the stability in Wasserstein distance $\mathcal{W}_1$ was also proved in \cite{FS} which means that the solution of \eqref{meq} is unique in $P_1(\mathbb{R}^d)$ for $t\in[0,T]$. 
For the exponential stability later, we discuss solutions $\mu(t,x)$ in $C([0,T];P_2(\mathbb{R}^d))$ with metric $\mathcal{W}_2$.
By adapting the method in \cite{CCR,FS}, we have

\begin{lemma}\label{lemma-W2-stability}
For fixed control $u(t,x)$, if $\mu(t,x)$, $\nu(t,x)$ are solutions to \eqref{meq} with initial data $\mu_0$ and $\nu_0$ satisfying the assumption in Theorem \ref{thm21}, then there is a constant $C > 0$ such that
$$\mathcal{W}_2(\mu(t,\cdot),\nu(t,\cdot)) \leq e^{Ct}~\mathcal{W}_2(\mu_0,\nu_0)\quad \text{for} \quad t \in [0, T].
$$
\end{lemma}
\noindent 
Some remarks are in order. The proof is similar to  \cite{CCR,FS} for the stability in $\mathcal{W}_1$ and deferred to the Appendix \ref{sectionA}. 
Hence, the optimal solution is unique in $C([0,T];P_2(\mathbb{R}^d))$ if the initial data $\mu_0\in P_2(\mathbb{R}^d)$. Due to this argument, we assume that the optimal solution $\mu(t,x)$ also satisfies
\begin{equation}
\lim_{k\rightarrow \infty}\mathcal{W}_2(\mu_{N_k}(t,\cdot), \mu(t,\cdot)) = 0
\end{equation}
uniformly with respect to $t \in [0, T]$. The assumption is justified since we have the convergence in $\mathcal{W}_1$  and the uniform boundness of the second order moment for $\mu_N(t,\cdot)$ with respect to  $N$, see e.g.  Theorem \ref{thm14}. 

% each $x_{i0}, x_i(t), u_i(t)\in \mathbb{R}^d$. 
% In particular, the control $u_i(t)=u_i(t,x(t))$ depends on the vector $x(t)$.

% Next we also assume the strict dissipativity and lower semi-continuity for the cost function. Namely,

% Now we would like to recall the theory in \cite{FS} and \cite{ACFK}. 
% Let $(u_{N},\mu_{N})$ be the solution to the optimal control problem of the ODE system with $N$ particles. Then there is a sub-sequence 
% $u_{N_k}\rightharpoonup u_{\infty}$ in $\mathcal{F}$ (see Lemma 2.2 in \cite{ACFK}). Denote $\mu_{\infty}(t,x)$  the solution to the mean-field equation \eqref{meq} with the control $u_{\infty}(t,x)$. It can be shown that $(\mu_{\infty},u_{\infty})$ is the solution to the optimal control problem in the mean-field level. Particularly, 
% In the present work, we will also use this property to obtain the estimate for the mean-field optimal control problem.

\section{Cheap control property}\label{Section2.1}
% We know that under the dissipativity and cheap control condition, it holds that
% \begin{align*}
% \frac{1}{N}\int_a^T \sum_{k=1}^N |x_k(t)-\bar{x}|^2dt  + \frac{1}{N}\int_a^T \sum_{k=1}^N |u_k(t)-\bar{u}|^2dt  
% \leq&~ C_0 \frac{1}{N}\sum_{k=1}^N |x_i(a)-\bar{x}|. 
% \end{align*}
% From now on, we will take $\bar{u}$ to be zero such that $x_i(t)=\bar{x}$ and $u(t)=\bar{u}$ is a static pair. 

% \begin{remark}\label{rem1}
%     The Assumption \ref{ass1} implies that the cost function is equivalent to the quadratic form in the sense:
% \begin{align*}
%         C_D \frac{1}{N}\sum_{k=1}^N\int_a^b |x_k(t)-\bar{x}|^2 + |u_k(t)|^2 dt \leq  \int_a^b f_N(x(t),u(t)) dt \\[2mm]\leq (C_L+C_{\Psi})\frac{1}{N}\sum_{k=1}^N\int_a^b |x_k(t)-\bar{x}|^2 + |u_k(t)|^2 dt.
% \end{align*}
% \end{remark}

The cheap control property of the optimal control problem shows that the optimal values are bounded by the distance between the initial state and the desired static state. Combining the cheap control property with the strict dissipativity, we  provide a bound on the second--order moments of the proability density. More specifically, for the $N$-particles system 
\eqref{particle-pb}, we prove:

\begin{lemma}\label{lemma11}
    Suppose $u_N$ is an optimal control to the problem $\mathcal{Q}_N(0,T,x_0)$ and $x(t)$ is the corresponding solution, then $u_N|_{t\in [a,T]}$ is also an optimal control to the sub-problem $\mathcal{Q}_N(a,T,x(a))$ for any $a\geq 0$. Moreover, the following inequality holds under Assumption \ref{ass1}:
    \begin{align}
\frac{1}{N} \sum_{i=1}^N \int_a^T |x_i(t)-\bar{x}|^2 +   |u_N(t,x_i(t))|^2dt  
\leq C_0 \frac{1}{N}\sum_{i=1}^N |x_i(a)-\bar{x}|^2.   \label{cheap} 
\end{align}
Here, $C_0$ is a positive constant independent of $N$.
\end{lemma}
\begin{proof}
    Suppose there exists a control $\tilde{u}_N$, defined on $t\in[a,T]$, such that the corresponding solution $\tilde{x}(t)$ satisfies $\tilde{x}(a)=x(a)$ and 
    $$
    \int_a^T f(\tilde{\mu}_{N},\tilde{u}_N) dt < \int_a^T f_N(\mu_N,u_N) dt.
    $$
    Here $\tilde{\mu}_{N}$ is the empirical measure given by
    $$
\tilde{\mu}_{N} = \frac{1}{N}\sum_{i=1}^N \delta\left(x-\tilde{x}_i(t)\right).
$$
Then, we  construct a control 
$$
\hat{u}_N(t,x) = 
\left\{
\begin{array}{ll}
    u_N(t,x), &  \qquad t\in[0,a) \\[2mm]
    \tilde{u}_N(t,x), &  \qquad t\in[a,T].
\end{array}
\right.
$$
In this case, the cost satisfies 
    \begin{align*}
        \int_0^T f(\hat{\mu}_{N},\hat{u}_N) dt = \int_0^a f(\mu_{N},u_N) dt + \int_a^T f(\tilde{\mu}_{N},\tilde{u}_N) dt 
        <  \int_0^T f(\mu_{N},u_N) dt .
    \end{align*}
    This contradicts to the fact that $(x(t),u_N(t))$ is an optimal solution on $[0,T]$. Therefore, $u_N|_{t\in[a,T]}$ is an optimal control for the sub-problem $\mathcal{Q}_N(a,T,x(a))$.

Thanks to the strict dissipativity, we have
\begin{align*}
        \int_a^T f(\mu_N,u_N) dt \geq&~ C_D 
        \int_a^T \int_{\mathbb{R}^d}\Big(|x-\bar{x}|^2 +|u_N(t,x)|^2 \Big) d\mu_N(t,x)dt \\[2mm]
        =&~C_D \frac{1}{N}\sum_{i=1}^N\int_a^T |x_i(t)-\bar{x}|^2 + |u_N(t,x_i(t))|^2 dt.
\end{align*}
By Remark \ref{remark21}, we obtain  the estimate \eqref{cheap}  once we prove the following cheap control inequality:
\begin{equation}\label{cheap1}
    \int_a^T f(\mu_N,u_N) dt = \frac{1}{N}\sum_{i=1}^N\int_{a}^T L(x_i(t)) + \Psi(u_N(t,x_i(t))) dt \leq \tilde{C}_0 \frac{1}{N}\sum_{i=1}^N |x_i(a)-\bar{x}|^2
\end{equation}
for a constant $\tilde{C}_0>0$ independent of $N$.

Next, we focus on the proof of \eqref{cheap1}. To this end, we consider the feedback control for the problem \eqref{particle-pb}:
$$
\tilde{u}_N(t,\tilde{x}_i(t))=-\beta(\tilde{x}_i(t)-\bar{x})-\frac{1}{N}\sum_{j=1}^NP(\tilde{x}_i(t)-\tilde{x}_j(t)),\qquad i=1,2,,...,N,\quad t\in[a,T].
$$
Note that  $\tilde{u}_N\in\mathcal{F}$ holds. Indeed, due to assumption \eqref{Lip-P}, we have that 
\begin{align*}
|\tilde{u}_N(t,x)-\tilde{u}_N(t,y)|
=&~ \Big|\beta(x-y)+\frac{1}{N}\sum_{j=1}^N\big[ P(x-\tilde{x}_j(t)) -P(y-\tilde{x}_j(t)) \big]\Big|\\[2mm]
\leq&~ \beta|x-y| + C_P \frac{1}{N} \sum_{j=1}^N |x-y| = (\beta + C_P)|x-y|,
\end{align*}
which gives a Lipschitz constant for $\tilde{u}_N(t,\cdot)$. 
Based on this feedback control, $\tilde{x}_i(t)$ satisfies the equation
$$
\frac{d\tilde{x}_i(t)}{dt} = -\beta (\tilde{x}_i(t)-\bar{x}),\qquad \tilde{x}_i(a) = x_i(a). 
$$
It follows that 
\begin{equation}\label{com1}
|\tilde{x}_i(t)-\bar{x}|^2 = e^{-2\beta (t-a)}|\tilde{x}_i(a)-\bar{x}|^2 = e^{-2\beta (t-a)}|x_i(a)-\bar{x}|^2.
\end{equation}
In the next paragraph, we estimate $|\tilde{u}_N(t,\tilde{x}_i(t))|^2$. By definition, we have
$$
|\tilde{u}_N(t,\tilde{x}_i(t))|^2 \leq 2\beta^2 |\tilde{x}_i(t)-\bar{x}|^2 + 2\Big|\frac{1}{N}\sum_{j=1}^NP(\tilde{x}_i(t)-\tilde{x}_j(t))\Big|^2.
$$
Using Jensen's inequality,
% $$
% \Big|\sum_{j=1}^N a_j\Big|^2 \leq N\sum_{j=1}^N |a_j|^2,
% $$
we have
\begin{eqnarray}\label{cauchy}
\Big|\frac{1}{N}\sum_{j=1}^NP(\tilde{x}_i(t)-\tilde{x}_j(t))\Big|^2 
\leq \frac{1}{N}\sum_{j=1}^N\Big|P(\tilde{x}_i(t)-\tilde{x}_j(t))\Big|^2.    
\end{eqnarray}
Due to the assumption of 
$P(x)$, we have
\begin{align*}
\frac{1}{N}\sum_{j=1}^N\Big|P(\tilde{x}_i(t)-\tilde{x}_j(t))\Big|^2
\leq&~\frac{C_P^2}{N}\sum_{j=1}^N|\tilde{x}_i(t)-\tilde{x}_j(t)|^2\\[2mm]
\leq&~2C_P^2|\tilde{x}_i(t)-\bar{x}|^2+
\frac{2C_P^2}{N}\sum_{j=1}^N|\tilde{x}_j(t)-\bar{x}|^2.
\end{align*}
Then, it follows that
\begin{align*}
    |\tilde{u}_N(t,\tilde{x}_i(t))|^2 \leq &~(2\beta^2+4C_P^2)|\tilde{x}_k(t)-\bar{x}|^2 + \frac{4C_P^2}{N}\sum_{j=1}^N|\tilde{x}_j(t)-\bar{x}|^2.
\end{align*}    
We sum $i$ from $1$ to $N$ and get 
\begin{align*}
   \frac{1}{N}\sum_{i=1}^N|\tilde{u}_N(t,\tilde{x}_i(t))|^2 \leq & ~C(\beta,C_P) \frac{1}{N}\sum_{i=1}^N |\tilde{x}_i(t)-\bar{x}|^2
\end{align*}
with $C(\beta,C_P)=2\beta^2+8C_P^2$. Since $u_N$ is optimal in \eqref{particle-pb}, we have 
\begin{align*}
    \frac{1}{N}\sum_{i=1}^N\int_{a}^T L(x_i(t)) + \Psi(u_N(t,x_i(t))) dt 
\leq &~
\frac{1}{N}\sum_{i=1}^N\int_a^T L(\tilde{x}_i(t))+\Psi(\tilde{u}_N(t,\tilde{x}_i(t)))dt \\[2mm]
\leq&~ (C({\beta},C_P)C_{\Psi}+C_L)\frac{1}{N}\sum_{i=1}^N \int_a^T |\tilde{x}_i(t)-\bar{x}|^2 dt.
\end{align*}
Note that the last inequality is due to Assumption~\ref{ass1} (ii).
Substituting \eqref{com1} into the last inequality, we have
\begin{align*}
    &~\frac{1}{N}\sum_{i=1}^N\int_{a}^T L(x_i(t)) + \Psi(u_N(t,x_i(t))) dt \\[2mm]
\leq &~(C({\beta},C_P)C_{\Psi}+C_L)\left(\int_a^Te^{-2\beta(t-a)}dt\right)\frac{1}{N}\sum_{i=1}^N |x_i(a)-\bar{x}|^2.
\end{align*}
It is easy to show that
$$
\int_a^Te^{-2\beta(t-a)}dt=\frac{1}{2\beta}e^{-2\beta(t-a)}\Big|_T^a\leq \frac{1}{2 \beta}.
$$ 
Then, we conclude  
\begin{align}\label{cheap-step1}
\frac{1}{N}\sum_{i=1}^N\int_{a}^T L(x_i(t)) + \Psi(u_N(t,x_i(t))) dt
\leq \tilde{C}_0\frac{1}{N}\sum_{i=1}^N |x_i(a)-\bar{x}|^2.
\end{align}
Note that the constant $\tilde{C}_0 = [C({\beta},C_P)C_{\Psi}+C_L]/(2\beta)$ is independent of $N$. 
\end{proof}

The estimate \eqref{cheap} is independent of $N$. We consider $N\rightarrow \infty$ to get the corresponding result for the mean-field problem. To this end, we also need to use the lower semi-continuity of the cost function \eqref{mec}. Namely, we prove the following property for the mean-field problem. 

\begin{lemma}
Suppose $(\mu(t,x),u(t,x))$ is the  solution to the optimal control problem \eqref{mec}-\eqref{meq}, then the following inequality holds under Assumption \ref{ass1}:
\begin{align}
\int_a^T \int_{\mathbb{R}^d}\Big(|x-\bar{x}|^2 +|u(t,x)|^2 \Big) d\mu(t,x)dt 
\leq &~ C_0 \int |x-\bar{x}|^2 d\mu(a,x).\label{cheap-w} 
\end{align}    
\end{lemma}

\begin{proof}
% \textbf{Step 2:(Mean-field limit).}
Due to lower semi-continuity, we have 
\begin{align*}
\int_a^T f(\mu(t,x),u(t,x))dt \leq&~ \liminf_{k\rightarrow \infty}\int_a^T f(\mu_{N_k}(t,x),u_{N_k}(t,x))dt \\[2mm] 
= &~ \liminf_{k\rightarrow \infty}\frac{1}{N_k}\sum_{i=1}^{N_k}\int_a^T L(x_i(t))+\Psi(u_{N_k}(t,x_i(t)))dt. 
\end{align*}
On the other hand, since $u_{N_k}$ is the optimal solution to \eqref{particle-pb}, it follows from \eqref{cheap1} that
\begin{align*}
\int_a^T f(\mu(t,x),u(t,x))dt \leq &~\liminf_{k\rightarrow \infty}\frac{1}{N_k}\sum_{i=1}^{N_k}\int_a^T L(x_i(t))+\Psi(u_{N_k}(t,x_i(t)))dt\\[2mm]
% \leq&~ \liminf_{k\rightarrow \infty}\frac{1}{N_k}\sum_{i=1}^{N_k}\int_a^b L(\tilde{x}_i(t))+\Psi(u_{N_k}(t,x_i(t)))dt  \\[2mm]
\leq&~ \liminf_{k\rightarrow \infty} \tilde{C}_0\frac{1}{N_k}\sum_{i=1}^{N_k} |x_i(a)-\bar{x}|^2 \\[2mm]
=&~ \tilde{C}_0 \int |x-\bar{x}|^2 d\mu(a,x).
\end{align*}
Here, $\tilde{C}_0$ is the constant introduced in Lemma \ref{lemma11}. 
Using the strict dissipativity shows that
\begin{align*}
C_D 
\int_a^T \int_{\mathbb{R}^d}\Big(|x-\bar{x}|^2 +|u(t,x)|^2 \Big) d\mu(t,x)dt 
\leq \int_a^T f(\mu(t,x),u(t,x))dt 
\leq \tilde{C}_0 \int |x-\bar{x}|^2 d\mu(a,x).
\end{align*}
Thus we can take $C_0=\tilde{C}_0/C_D$ to conclude the result.
\end{proof} 

We conclude this section with the following remarks: 
\begin{itemize}
    \item The inequality \eqref{cheap-w} is the mean-field limit of relation \eqref{cheap}.
    \item The right-hand side of \eqref{cheap-w} is independent of $T$. As in other turnpike results, this shows an integral turnpike property. Namely, the second order moments $\int_{\mathbb{R}^d}\Big(|x-\bar{x}|^2 +|u(t,x)|^2 \Big) d\mu(t,x)$ must be small along the largest part of the time-horizon provided that $T$ is sufficiently large. 
    \item  The cheap control idea was also used in \cite{main} to prove the integral turnpike property with interior decay. Different from the results in  \cite{main}, the present work uses the second-order moment $\int_{\mathbb{R}^d} |x-\bar{x}|^2 d\mu(a,x)$ as the bound in \eqref{cheap-w} instead of the first-order moment. This is important for the proofs in the next section.
\end{itemize}
 
\section{Exponential turnpike property}\label{Section 4}

In this section, we will prove that the optimal solution to \eqref{mec}-\eqref{meq} converges to the optimal static state exponentially fast. The estimates on the inequalities for the optimal solution $\mu(t,x)$ and the optimal control $u(t,x)$ are given separately, see Theorem \ref{theorem-m-s} and \ref{thm14-c} below.  
% The proof basically consists of three steps. 
To this end, we derive the estimate for the optimal solution $x_i(t)$  of the $N$-particles system. Then, we consider the mean-field limit $N \rightarrow \infty$ to obtain an estimate for $\mu(t,x)$. At last, we prove that the optimal control $u(t,x)$ can be bounded in terms of the solution $\mu(t,x)$.

\subsection{Estimate for the solution}\label{Section 4-1}

For the solution $x_i(t)$ of \eqref{particle-pb}, we use Gronwall's inequality to derive 

\begin{lemma}\label{lemma12}
Suppose \eqref{cheap} holds, there exists a constant $C_1\geq 1$ such that
\begin{equation}
    \frac{1}{N} \sum_{i=1}^N |x_i(t_2)-\bar{x}|^2 
\leq C_1 \frac{1}{N}\sum_{i=1}^N |x_i(t_1)-\bar{x}|^2,\qquad \forall~0 \leq t_1 \leq t_2 \leq T.
\end{equation}
\end{lemma}
\begin{proof}
We estimate $y_i(t)=x_i(t)-\bar{x}$ by computing:
\begin{align}
    \frac{1}{2}\int_{t_1}^{t_2}\frac{d}{dt}\langle y_i(t), y_i(t) \rangle dt=&~\int_{t_1}^{t_2}\langle y_i(t),y_i'(t) \rangle dt \nonumber
    \\[2mm]
    = &~ \frac{1}{N}\sum_{j=1}^N \int_{t_1}^{t_2} \langle y_i(t), P(y_i(t)-y_j(t)) \rangle dt + \int_{t_1}^{t_2} \langle y_i(t), u_i(t) \rangle dt. \label{thm12eq1}
\end{align}
For the second term, we have
\begin{eqnarray}\label{thm12eq2}
\int_{t_1}^{t_2} \langle y_i(t), u_i(t) \rangle dt \leq \frac{1}{2} \int_{t_1}^{t_2} |u_i(t)|^2 dt + \frac{1}{2} \int_{t_1}^{t_2} |y_i(t)|^2 dt,    
\end{eqnarray}
and for the first term, we have
\begin{align}
    &\frac{1}{N}\sum_{j=1}^N \int_{t_1}^{t_2} \langle y_i(t), P(y_i(t)-y_j(t)) \rangle dt \leq ~ \frac{1}{N}\sum_{j=1}^N C_P \int_{t_1}^{t_2} | y_i(t)||y_i(t)-y_j(t)| dt \nonumber \\[2mm]
    \leq ~& \frac{1}{N}\sum_{j=1}^N C_P \int_{t_1}^{t_2} | y_i(t)|^2+|y_i(t)||y_j(t)| dt  
    \leq ~ \frac{3C_P}{2} \int_{t_1}^{t_2} | y_i(t)|^2dt + \frac{C_P}{2N}\sum_{j=1}^N  \int_{t_1}^{t_2} |y_j(t)|^2 dt.\label{thm12eq3}
\end{align}
Combining \eqref{thm12eq1}---\eqref{thm12eq3} yields
$$
\frac{1}{2}\int_{t_1}^{t_2}\frac{d}{dt}\langle y_i(t), y_i(t) \rangle dt \leq \Big(\frac{1}{2}+\frac{3C_P}{2}\Big) \int_{t_1}^{t_2} |y_i(t)|^2 dt 
+ \frac{C_P}{2N}\sum_{j=1}^N  \int_{t_1}^{t_2} |y_j(t)|^2 dt
+ \frac{1}{2}\int_{t_1}^{t_2} |u_i(t)|^2 dt.
$$
We sum $i$ from $1$ to $N$ and multiply $1/N$ to obtain
% $$
% \frac{1}{2N}\sum_{i=1}^N\int_{t_1}^{t_2}\frac{d}{dt}\langle y_i(t), y_i(t) \rangle dt \leq \Big(2C_P+\frac{1}{2}\Big)\frac{1}{N} \sum_{i=1}^N \int_{t_1}^{t_2} |y_i(t)|^2 dt + \frac{1}{2N} \sum_{i=1}^N \int_{t_1}^{t_2} |u_i(t)|^2 dt
% $$
% which implies
$$
\frac{1}{N}\sum_{i=1}^N|y_i(t_2)|^2 \leq \frac{1}{N}\sum_{i=1}^N|y_i(t_1)|^2 + \left(1+4C_P\right) \frac{1}{N} \sum_{i=1}^N \int_{t_1}^{t_2} |y_i(t)|^2 dt + \frac{1}{N} \sum_{i=1}^N \int_{t_1}^{t_2} |u_i(t)|^2 dt.
$$
Combining this with \eqref{cheap}, we obtain
\begin{equation*}
    \frac{1}{N} \sum_{i=1}^N |x_i(t_2)-\bar{x}|^2 
\leq C_1 \frac{1}{N}\sum_{i=1}^N |x_i(t_1)-\bar{x}|^2,\qquad \forall~0 \leq t_1 \leq t_2 \leq T.
\end{equation*}
with 
$C_1= (2+4C_P)C_0+1$.
\end{proof}

Combining this lemma with the inequality \eqref{cheap}, we prove:

\begin{lemma}\label{lemma13}
Under Assumption \ref{ass1}, the following inequality holds for any $t \in [n\tau,T]$ with a given constant $\tau>0$ and an integer $1\leq n\leq \frac{T}{\tau}$:
$$
\frac{1}{N}\sum_{i=1}^N|x_i(t)-\bar{x}|^2 \leq \left(\frac{C_0C_1}{\tau}\right)^n \frac{1}{N}\sum_{i=1}^N|x_i(0)-\bar{x}|^2.
$$ 
\end{lemma}
\begin{proof}
    We first prove the case  $n=1$. There exists a point $t_1\in[0,\tau]$ such that
    \begin{align*}
    \frac{1}{N}\sum_{i=1}^N|x_i(t_1)-\bar{x}|^2 \leq \frac{1}{\tau} \int_{0}^{\tau}
    \frac{1}{N}\sum_{i=1}^N|x_i(t)-\bar{x}|^2 dt
    \leq \frac{C_0}{\tau} \frac{1}{N}\sum_{i=1}^N|x_i(0)-\bar{x}|^2.
    \end{align*}
    Note that the last inequality follows by \eqref{cheap}. For any $t\geq \tau \geq t_1$, we obtain by Lemma \ref{lemma12}
    \begin{align*}
    \frac{1}{N}\sum_{i=1}^N|x_i(t)-\bar{x}|^2
    \leq C_1\frac{1}{N}\sum_{i=1}^N|x_i(t_1)-\bar{x}|^2
    \leq \frac{C_0C_1}{\tau} \frac{1}{N}\sum_{i=1}^N|x_i(0)-\bar{x}|^2.
    \end{align*}
    
    Then we suppose the inequality holds for $n\geq 1$ and prove the result for $n+1$. There exists $t_n\in[n\tau,(n+1)\tau]$ such that
\begin{align*}
    \frac{1}{N}\sum_{i=1}^N|x_i(t_n)-\bar{x}|^2 \leq&~ \frac{1}{\tau} \int_{n\tau}^{(n+1)\tau}\frac{1}{N}\sum_{i=1}^N|x_i(t)-\bar{x}|^2dt\\[2mm] 
    \leq&~\frac{C_0}{\tau} \frac{1}{N}\sum_{i=1}^N|x_i(n\tau)-\bar{x}|^2\leq \frac{C_0}{\tau} \left(\frac{C_0C_1}{\tau}\right)^n \frac{1}{N}\sum_{i=1}^N|x_i(0)-\bar{x}|^2.
    \end{align*}
Thus for any $t\in[(n+1)\tau,T]$, we obtain by Lemma \ref{lemma12}
$$
\frac{1}{N}\sum_{i=1}^N|x_i(t)-\bar{x}|^2 \leq C_1\frac{1}{N}\sum_{i=1}^N|x_i(t_n)-\bar{x}|^2\leq \left(\frac{C_0C_1}{\tau}\right)^{n+1} \frac{1}{N}\sum_{i=1}^N|x_i(0)-\bar{x}|^2
$$ 
and this completes the proof.
\end{proof}

Thanks to the above lemmas, we are in the position to state the main result for the optimal solution $x_i(t)$ of the particle system \eqref{particle-pb}:

\begin{theorem}\label{thm14}
    Suppose Assumption \ref{ass1} holds. Then there exist constants $C_2>0$ and $\alpha>0$ such that the optimal solution for $\mathcal{Q}_N(0,T,x_0)$ satisfies the exponential turnpike property:
    $$
    \frac{1}{N}\sum_{i=1}^N|x_i(t)-\bar{x}|^2 \leq C_2 e^{-\alpha t} \frac{1}{N}\sum_{i=1}^N|x_i(0)-\bar{x}|^2
    $$
    for any $t\in(0,T)$.
\end{theorem}

\begin{proof}
    % First we prove the estimate for the optimal solution $x_k^*(t)$ by resorting to Lemma \ref{lemma13}. 
    In this proof, we need to fix the constant $\tau$ in Lemma \ref{lemma13} such that $\tau>C_0C_1$. Next, we discuss the cases $t\in(0,\tau)$ and $t\in[\tau,T)$ separately. 
    
    For any $t\in[\tau,T)$, we take the integer $n=\lfloor t/\tau \rfloor$. Then,  $1\leq n \leq \frac{T}{\tau}$  and $t\in [n\tau,T)$
    and we obtain by Lemma \ref{lemma13}:
\begin{align*}
    \frac{1}{N}\sum_{i=1}^N|x_i(t)-\bar{x}|^2 \leq &~ \left(\frac{C_0C_1}{\tau}\right)^n \frac{1}{N}\sum_{i=1}^N|x_i(0)-\bar{x}|^2.
\end{align*}
Due to the definition of $n$, we have  $n>t/\tau-1$. Also, the constant $\tau$ is chosen such that $\tau>C_0C_1$. Thus we have
\begin{align*}
    &~ \left(\frac{C_0C_1}{\tau}\right)^n = \left(\frac{\tau}{C_0C_1}\right)^{-n} 
    \leq \left(\frac{\tau}{C_0C_1}\right)^{1-t/\tau}.
\end{align*}
The exponential estimate is then given by
$$
    \frac{1}{N}\sum_{i=1}^N|x_i(t)-\bar{x}|^2 \leq \hat{C}_2 e^{-\alpha t} \frac{1}{N}\sum_{i=1}^N|x_i(0)-\bar{x}|^2,\quad \forall ~t\in[\tau, T)
$$
with 
$$
 \hat{C}_2 = \frac{\tau}{C_0C_1},\qquad \alpha = \frac{1}{\tau}\log\left(\frac{\tau}{C_0C_1}\right)>0.
$$

On the other hand, for $t\in(0,\tau)$, we have
$$
 \hat{C}_2 e^{-\alpha t} \geq  \hat{C}_2 e^{-\alpha \tau} = 1.
$$
By Lemma \ref{lemma12} we have 
\begin{equation*}
    \frac{1}{N} \sum_{i=1}^N |x_i(t)-\bar{x}|^2 
\leq C_1 \frac{1}{N}\sum_{i=1}^N |x_i(0)-\bar{x}|^2 \leq C_1  \hat{C}_2 e^{-\alpha t}\frac{1}{N}\sum_{i=1}^N |x_i(0)-\bar{x}|^2.
\end{equation*}
Recall, that due to the proof of Lemma \ref{lemma12},  $C_1\geq 1$ holds. To combine the results of $t\in(0,\tau)$ and $t\in[\tau, T)$, we  take $C_2=C_1 \hat{C}_2$ and obtain
\begin{equation}\label{thm14-eq4}
\frac{1}{N} \sum_{i=1}^N |x_i(t)-\bar{x}|^2 
\leq C_2 e^{-\alpha t}\frac{1}{N}\sum_{i=1}^N |x_i(0)-\bar{x}|^2,\quad \forall~t\in(0,T).
\end{equation}
\end{proof}

This theorem implies that the empirical measure has equi-compact support and bounded second-order moments for any number of particles $N$. Moreover, we know that the empirical measure $\mu_N(t,x)$ defined in \eqref{empirical} satisfies
$$
\mathcal{W}_2(\mu_N(t,\cdot),\delta(x-\bar{x})) \leq \sqrt{C_2} e^{-\alpha t/2} \mathcal{W}_2(\mu_N(0,\cdot),\delta(x-\bar{x})).
$$

We established the exponential decay property for the second-order moment of the empirical measures $\mu_N(t,\cdot)$ with respect to $t$:
$$
    \int |x-\bar{x}|^2 d\mu_N(t,x) \leq C_2 e^{-\alpha t} \int |x-\bar{x}|^2 d\mu_N(0,x).
$$
The constant $C_2$ is independent of $N$. Thus, we can use the uniform $\mathcal{W}_2$ convergence to obtain the exponential turnpike property in the mean field limit. Namely, we have
\begin{theorem}\label{theorem-m-s}
    Suppose Assumption \ref{ass1} holds. For problem $\mathcal{Q}(0,T,\mu_0)$, the optimal solution $\mu(t,x)\in C([0,T];P_2(\mathbb{R}^d))$ satisfies the exponential turnpike property in the sense that
    $$
    \int |x-\bar{x}|^2 d\mu(t,x) \leq C_2 e^{-\alpha t} \int |x-\bar{x}|^2 d\mu_0(x)
    $$
    for any $t\in(0,T)$. Here the constants $C_2$ and $\alpha$ are the same as those in Theorem \ref{thm14}.
\end{theorem}

\begin{remark}
Alternatively, the result on the mean field problem can be also proven by a direct estimate of \eqref{meq}. 
Namely, we may take a test function  $\phi(t,x)=|x-\bar{x}|^2\chi_R(x)$ with $\chi_R(x)$ being a mollified characteristic function $\chi_R(x)=\psi_\delta \ast \chi_{[-R-\delta,R+\delta]}$, such that $\chi_R(x)=1$ for $|x|\leq R$. 
%%% 

Then by the same argument as  in Lemma \ref{lemma12}, we have 
\begin{eqnarray*}
\begin{aligned}
    &\int |x-\bar{x}|^2 d\mu(t_2,x) \leq C_1 \int |x-\bar{x}|^2 d\mu(t_1,x),\qquad \forall~0\leq t_1 \leq t_2 \leq T.
\end{aligned}
\end{eqnarray*}
Similarly, the inequalities analog to those in Lemma \ref{lemma13} and Theorem \ref{thm14} can be also obtained.  
\end{remark}

\subsection{Estimate on the control}\label{Section 4-2}
In this subsection, we estimate the optimal control $u(t,x)$ in the mean field problem. The idea is to construct a novel feedback control and take advantage of the strict dissipativity. 
% At first, we state the estimate of $u_N$, which is the optimal control for the problem \eqref{particle-pb} with $N$ particles.
%\subsubsection{Feedback control}

We divide the time interval $[0,T]$ into three parts:
$$
[0,T] = [0,s)\cup [s,s+mh] \cup (s+mh,T].
$$
Here $s\in (0,T)$ is a fixed time point, $m>0$ is a scale parameter which will be given later (see \eqref{m}), and $h$ is a sufficiently small constant such that $s+mh\leq T$. We construct a feedback control $\hat{u}(t,x)$ by
\begin{equation}\label{feedback-est}
\hat{u}(t,x) = 
\left\{
\begin{array}{ll}
    u(t,x), &  \qquad t\in[0,s) \\[3mm]
    \dfrac{1}{m}u\left(s+\dfrac{t-s}{m},x\right)-\dfrac{m-1}{m} (P\ast \hat{\mu})(t,x) & \qquad t\in[s,s+mh] \\[5mm]
    u\Big(t-(m-1)h,x\Big), &  \qquad t\in(s+mh,T],
\end{array}
\right.
\end{equation}
where $u(t,x)$ is the optimal control to the problem \eqref{mec}-\eqref{meq} on the time interval $[0,T]$, $\hat{\mu}(t,x)$ is the solution of \eqref{meq} associated to the new control $\hat{u}(t,x)$, 

% and new time variables are defined as
% $$
% t_1=s+\frac{t-s}{m},\qquad t_2=t-(m-1)h.
% $$ 

Next, we  discuss the solution $\hat{\mu}(t,x)$ on the different time intervals. 

% $$
% \tilde{u}_N(t,\tilde{x}_i(t)) = 
% \left\{
% \begin{array}{ll}
%     u_N(t,\tilde{x}_i(t)), &  \qquad s\in[0,t) \\[2mm]
%     \dfrac{1}{m}u_N(s_1,\tilde{x}_i(s))-\dfrac{m-1}{m}\dfrac{1}{N}\sum_{j=1}^NP(\tilde{x}_i(s)-\tilde{x}_j(s)) & \qquad s\in[t,t+mh] \\[4mm]
%     u_N(s_2,\tilde{x}_i(s)), &  \qquad s\in(t+mh,T].
% \end{array}
% \right.
% $$
% with 
% Here $m$ is a fixed constant which will be given later (below formula \eqref{est1}). Since $t\in(0,T)$, we find a sufficiently small constant $h$ such that $t+mh\leq T$.

% We will discuss the solution $\tilde{x}_k(s)$ associated to the new control $\tilde{u}_N$ on different time intervals. 
For $t\in[0,s)$, we know that $\hat{u}(t,x)=u(t,x)$ and the initial data satisfies 
    $$
    \hat{\mu}(0,\cdot) = \mu_0(\cdot)\quad \text{in}~ P_2(\mathbb{R}^d).
    $$ 
    According to the uniqueness of solution to the mean field equation \eqref{meq}, it is easy to see that 
$$
\hat{\mu}(t,\cdot) = \mu(t,\cdot)\quad \text{in}~P_2(\mathbb{R}^d),\qquad \forall ~t\in[0,s].
$$
Here,  $\mu(t,x)$ is the solution associated to the optimal control $u(t,x)$.

On the other hand, for $t\in [s,s+mh]$, we use the expression of $\hat{u}(t,x)$ to compute the equation of $\hat{\mu}$ (for simplicity in the strong form). A similar computation holds in the weak form. 
% \begin{align}
%     \frac{d\tilde{x}_i}{ds} =&~ \frac{1}{N}\sum_{j=1}^NP(\tilde{x}_i(s)-\tilde{x}_j(s)) + \tilde{u}_N(s,\tilde{x}_i(s)) \nonumber 
%  \\[2mm]
%     =&~ \frac{1}{N}\sum_{j=1}^NP(\tilde{x}_i(s)-\tilde{x}_j(s)) + \dfrac{1}{m}u_N(s_1,\tilde{x}_i(s))-\dfrac{m-1}{m}\dfrac{1}{N}\sum_{j=1}^NP(\tilde{x}_i(s)-\tilde{x}_j(s)) \nonumber \\[2mm]
%     =&~ \dfrac{1}{m} \left(\frac{1}{N}\sum_{j=1}^NP(\tilde{x}_i(s)-\tilde{x}_j(s)) + u_N(s_1,\tilde{x}_i(s)) \right). \label{mideq}
% \end{align}
\begin{align}
    0=&~\partial_t \hat{\mu}(t,x) + \nabla_x \cdot \Big( \big[(P\ast \hat{\mu})(t,x)+\hat{u}(t,x)\big]\hat{\mu}(t,x)\Big)\nonumber 
 \\[2mm]
 =&~ \partial_t \hat{\mu}(t,x) + \nabla_x \cdot \Big( \Big[(P\ast \hat{\mu})(t,x)+\frac{1}{m} u\left(s+\dfrac{t-s}{m},x\right) - \frac{m-1}{m}(P\ast \hat{\mu})(t,x) \Big]\hat{\mu}(t,x)\Big) \nonumber \\[2mm]
    =&~ \partial_t \hat{\mu}(t,x) + \frac{1}{m} \nabla_x \cdot \Big( \Big[(P\ast \hat{\mu})(t,x)+u\left(s+\dfrac{t-s}{m},x\right)\Big]\hat{\mu}(t,x)\Big). \nonumber
\end{align}
Moreover, by the first step, we have  
$$
\hat{\mu}(s,\cdot) = \mu(s,\cdot)\quad \text{in}~P_2(\mathbb{R}^d). 
$$
Thus, the equation for $\hat{\mu}$  reads (in weak form)  
\begin{eqnarray}\label{mideq}
    \begin{aligned}
    &\int \phi(t,x)d\hat{\mu}(t,x) -\int \phi(s,x)d\mu(s,x) \\[2mm]
    = &\int_{0}^{t} \int \Big[\partial_t \phi(r,x) + \frac{1}{m} \nabla_x \phi(r,x) \cdot \Big((P\ast \hat{\mu})(r,x) + u\left(s+\dfrac{r-s}{m},x\right)\Big) \Big]d\hat{\mu}(r,x)dr \\[2mm]
    &\quad \forall~\phi(t,x)\in C_0^{\infty}\left([s,s+mh]\times \mathbb{R}^d\right).
\end{aligned}
\end{eqnarray}
% with 
% $$
% r_1 = s+\frac{r-s}{m}.
% $$
% According to the definition of $t_1$, we have
% \begin{equation}\label{s-s1}
% m\frac{\partial}{\partial t} = \frac{\partial}{\partial t_1}.    
% \end{equation}
Since the map 
$$
t_1 = s+\dfrac{t-s}{m}
$$
is bijective, we consider the test function
% By using the change of variables 
$$
\phi(t,x) = \hat{\phi}(t_1,x) =\hat{\phi}\left(s+\dfrac{t-s}{m},x\right)\quad \text{with}\quad 
\hat{\phi}\in C_0^{\infty}\left([s,s+h]\times \mathbb{R}^d\right)
$$
and the formula \eqref{mideq} is equivalent to 
\begin{eqnarray}\label{mideq-cv}
    \begin{aligned}
    &\int \hat{\phi}(t_1,x) d\hat{\mu}(t,x) -\int \hat{\phi}(s,x)d\mu(s,x) \\[2mm]
    =& \int_{0}^{t} \int \Big[\partial_{t} \hat{\phi}(r_1,x) + \nabla_x \hat{\phi}(r_1,x) \cdot \big((P\ast \hat{\mu})(r,x) + u(r_1,x)\big) \Big]d\hat{\mu}(r,x)dr_1, \\[2mm]
    &\forall~\hat{\phi}\in C_0^{\infty}\left([s,s+h]\times \mathbb{R}^d\right). 
\end{aligned}
\end{eqnarray}
Here, we use the relation 
$$
r_1=s+\dfrac{r-s}{m},\qquad dr_1=\frac{1}{m}dr, 
$$
and obtain that 
$$
\mu(t_1,x) = \mu\left(s+\dfrac{t-s}{m},x\right)
$$ 
is a solution to \eqref{mideq-cv}. Again,  $\mu(t,x)$ is the solution associated to the optimal control $u(t,x)$.
Since the solution for \eqref{meq} is unique in $P_2(\mathbb{R}^d)$, we have
\begin{equation}\label{est-step2}
\hat{\mu}(t,\cdot) = \mu(t_1,\cdot) = \mu\left(s+\dfrac{t-s}{m},\cdot\right) \quad \text{in}~P_2(\mathbb{R}^d),\qquad \forall ~t\in[s,s+mh].
\end{equation}
% Thus the trajectory $x_i(s_1)$ (the original optimal solution with the change of variables $s\rightarrow s_1$) satisfies the ODE \eqref{mideq} for each $i=1,2,...,N$. Moreover, from the discussion of the first time interval, we see that 
% $$
% x_i(s_1)|_{s_1=t} =x_i(t) = \tilde{x}_i(s)|_{s=t},\qquad i=1,2,...,N.
% $$
% It means that the trajectory $x_i(s_1)$ also satisfies the same initial condition as $\tilde{x}_i(s)$ at time $s=t$. Due to the uniqueness of ODE, we have 
% $$
% \tilde{x}_i(s)=x_i(s_1),\qquad s\in[t,t+mh].
% $$
In the last interval, for $t\in (s+mh,T]$, the control is $\hat{u}(t,x)=u\big(t-(m-1)h,x\big)$ and the equation for $\hat{\mu}$ reads (in strong form): 
\begin{align}
    0=&~\partial_t \hat{\mu}(t,x) + \nabla_x \cdot \Big( \big[(P\ast \hat{\mu})(t,x)+\hat{u}(t,x)\big]\hat{\mu}(t,x)\Big)\nonumber 
 \\[2mm]
 =&~ \partial_t \hat{\mu}(t,x) + \nabla_x \cdot \Big( \Big[(P\ast \hat{\mu})(t,x)+ u\big(t-(m-1)h,x\big) \Big]\hat{\mu}(t,x)\Big).\nonumber
\end{align}
Considering $t=s+mh$, we have 
$$
\hat{\mu}(s+mh,\cdot) = \mu(s+h,\cdot)\quad \text{in}~P_2(\mathbb{R}^d). 
$$
Thus, the weak form  in the time interval $(s+mh,T]$ reads as 
\begin{eqnarray}\label{mideq2}
    \begin{aligned}
    &\int \phi(t,x)d\hat{\mu}(t,x) -\int \phi(s+mh,x)d\mu(s+h,x) \\[2mm]
    = &\int_{0}^{t} \int \Big[\partial_t \phi(r,x) + \nabla_x \phi(r,x) \cdot \Big((P\ast \hat{\mu})(r,x) + u\big(t-(m-1)h,x\big) \Big]d\hat{\mu}(r,x)dr \\[2mm]
    &\quad \forall~\phi(t,x)\in C_0^{\infty}\left((s+mh,T]\times \mathbb{R}^d\right).
\end{aligned}
\end{eqnarray}
In the new variable 
$
t_2 = t-(m-1)h
$
and for the test function
$$
\phi(t,x) = \hat{\phi}(t_2,x) =\hat{\phi}\left(t-(m-1)h,x\right)\quad \text{with}\quad 
\hat{\phi}\in C_0^{\infty}\left((s+h,T-(m-1)h]\times \mathbb{R}^d\right),
$$
equation \eqref{mideq2} reads
\begin{eqnarray}\label{mideq2-cv}
    \begin{aligned}
    &\int \hat{\phi}(t_2,x) d\hat{\mu}(t,x) -\int \hat{\phi}(s+mh,x)d\mu(s+h,x) \\[2mm]
    = &\int_{0}^{t} \int \Big[\partial_{t} \hat{\phi}(r_2,x) + \nabla_x \hat{\phi}(r_2,x) \cdot \big((P\ast \hat{\mu})(r,x) + u(r_2,x)\big) \Big]d\hat{\mu}(r,x)dr_2, \\[2mm]
    &\quad \forall~\hat{\phi}\in C_0^{\infty}\left((s+h,T-(m-1)h]\times \mathbb{R}^d\right)
\end{aligned}
\end{eqnarray}
for $ r_2=r-(m-1)h$ and  $dr_2=dr.$
It is easy to see that 
$
\mu(t_2,x) = \mu\left(t-(m-1)h,x\right)
 $
satisfies \eqref{mideq2-cv}.
At last, we use the uniqueness of \eqref{meq} in $P_2(\mathbb{R}^d)$ to conclude that
\begin{equation}
\hat{\mu}(t,\cdot) = \mu(t_2,\cdot) = \mu\left(t-(m-1)h,\cdot\right) \quad \text{in}~P_2(\mathbb{R}^d),\qquad \forall ~t\in(s+mh,T].
\end{equation}
% we write the ODE system for $\tilde{x}_i(s)$ as 
% \begin{eqnarray}\label{lasttime}
%     \frac{d\tilde{x}_i}{ds} = \frac{1}{N}\sum_{j=1}^NP(\tilde{x}_i(s)-\tilde{x}_j(s)) + u_N(s_2,\tilde{x}_i(s)).
% \end{eqnarray}
% It is not difficult to check that $x_i(s_2)$ (the original optimal solution with the change of variables $s\rightarrow s_2$)
% satisfies the ODE \eqref{lasttime}. To check the initial condition at $s=t+mh$, we see from the definition of $s_2$ that 
% $$
% x_i(s_2)|_{s=t+mh}=x_i(s_2)|_{s_2=t+h}=x_i(t+h).
% $$
% From the previous discussion for $s\in [t,t+mh]$, we know that 
% $$
%  x_i(t+h) = x_i(s_1)|_{s=t+mh}  = \tilde{x}_i(t+mh) .
% $$
% Thus we obtain 
% $$
% \tilde{x}_i(s)=x_i^*(s_2),\qquad s\in(t+mh,T]
% $$ 
Summarizing, we have
\begin{equation}\label{proof-control-x}
\hat{\mu}(t,\cdot) = 
\left\{
\begin{array}{ll}
    \mu(t,\cdot), &  \qquad t\in[0,s), \\[4mm]
    \mu\left(s+\dfrac{t-s}{m},~\cdot\right), & \qquad t\in[s,s+mh],\\[5mm]
    \mu\left(t-(m-1)h,~\cdot\right), &  \qquad t\in(s+mh,T].
\end{array}
\right.
\end{equation}

\subsection{The turnpike estimate }
Having the feedback control $\hat{u}(t,x)$ and its associated solution $\hat{\mu}(t,x)$, we proceed to estimate the optimal control $u(t,x)$: 
\begin{theorem}\label{thm14-c}
    Suppose Assumption \ref{ass1} holds. Then there exists a constant $C_{3}>0$ such that the optimal control $u(t,x)\in \mathcal{F}$ for $\mathcal{Q}(0,T,\mu_0)$ satisfies the exponential turnpike property:
    $$
    \int |u(t,x)|^2 d\mu(t,x) \leq C_{3} e^{-\alpha t} \int |x-\bar{x}|^2 d\mu_0(x) \mbox{ for a.e. } t\in(0,T).
    $$
\end{theorem}
\begin{proof}
Since $u(t,x)$ is optimal, we have
\begin{align}
    &\int_{0}^Tf(\mu(t,x),u(t,x))dt
    \leq \int_{0}^Tf(\hat{\mu}(t,x),\hat{u}(t,x))dt \nonumber \\[2mm]
    =&~ \int_{0}^sf(\hat{\mu}(t,x),\hat{u}(t,x))dt + \int_s^{s+mh}f(\hat{\mu}(t,x),\hat{u}(t,x))dt + \int_{s+mh}^Tf(\hat{\mu}(t,x),\hat{u}(t,x))dt. \label{pr-eq1}
\end{align}
According to \eqref{feedback-est} and \eqref{proof-control-x}, we have  
\begin{equation}\label{pr-eq2}
    \int_{0}^sf(\hat{\mu}(t,x),\hat{u}(t,x))dt = \int_{0}^sf(\mu(t,x),u(t,x))dt
\end{equation}
and 
\begin{align}
    \int_{s+mh}^Tf(\hat{\mu}(t,x),\hat{u}(t,x))dt =~ \int_{s+mh}^Tf(\mu\left(t-(m-1)h,x\right),u\left(t-(m-1)h,x\right))dt \nonumber \\[2mm]
    =~ \int_{s+h}^{T-(m-1)h}f(\mu(t,x),u(t,x))dt 
    \leq ~ \int_{s+h}^{T}f(\mu(t,x),u(t,x))dt.\label{pr-eq3}
\end{align}
Therefore, it follows by \eqref{pr-eq1}-\eqref{pr-eq3}  
\begin{align}
    \int_s^{s+h}f(\mu(t,x),u(t,x))dt
    \leq &~ \int_s^{s+mh}f(\hat{\mu}(t,x),\hat{u}(t,x))dt \nonumber \\[2mm]
    \leq &~ C_4 \int_s^{s+mh}\int |x-\bar{x}|^2  + | \hat{u}(t,x)|^2 d\hat{\mu}(t,x)dt \label{prf-eq4}
\end{align}
with $C_4=\max\{C_{\Psi},C_L\}$.
% \begin{align*}
%     &\frac{1}{N}\in0^T\sum_{k=1}^N|x_k(s)-\bar{x}|^2+\sum_{k=1}^N|u_k(s)|^2 ds\leq \frac{1}{N}\in0^T\sum_{k=1}^N|\tilde{x}_k(s)-\bar{x}|^2+\sum_{k=1}^N|\tilde{u}_k(s)|^2 ds\\[2mm]
%     =&~\frac{1}{N}\in0^{T-h}\sum_{k=1}^N|x_k(s)-\bar{x}|^2ds + \frac{1}{N}\int_t^{t+h}\sum_{k=1}^N|x_k(s)-\bar{x}|^2ds + \frac{1}{N}\in0^{T-h}\sum_{k=1}^N|u_k(s)|^2ds\\[2mm]
%     &-\frac{1}{N}\int_t^{t+h}\sum_{k=1}^N|u_k(s)|^2 ds+\frac{2}{N}\int_t^{t+h}\sum_{k=1}^N|\frac{1}{2}u_k(s)-\frac{1}{2} \dfrac{1}{N}\sum_{j=1}^NP(x_k(s)-x_j(s))|^2ds.
% \end{align*}
Notice that the last inequality is due to Assumption \ref{ass1}. Moreover, we use \eqref{feedback-est} and \eqref{proof-control-x} to obtain
\begin{align*}
&~ C_4 \int_s^{s+mh}\int |x-\bar{x}|^2  + | \hat{u}(t,x)|^2 d\hat{\mu}(t,x)dt \\[2mm]
= &~ C_4 \int_s^{s+mh} \int | x-\bar{x}|^2 + | \frac{1}{m} u(t_1,x)- \frac{m-1}{m}(P\ast \mu)(t_1,x)|^2 d\mu(t_1,x)dt
\end{align*}
with $t_1=s+\dfrac{t-s}{m}$. By change of variables, the above inequality yields
\begin{align}
&~ C_4 \int_s^{s+mh}\int |x-\bar{x}|^2  + | \hat{u}(t,x)|^2 d\hat{\mu}(t,x)dt \nonumber \\[2mm]
\leq &~ m C_4 \int_s^{s+h} \int | x-\bar{x}|^2 + | \frac{1}{m} u(t,x)- \frac{m-1}{m}(P\ast \mu)(t,x)|^2 d\mu(t,x)dt \nonumber \\[2mm]
\leq &~ m C_4 \int_s^{s+h} \int | x-\bar{x}|^2 + \frac{3}{2}\frac{1}{m^2} |u(t,x)|^2 + 3\Big| \frac{m-1}{m}(P\ast \mu)(t,x)\Big|^2 d\mu(t,x)dt. \label{prf-eq5}
\end{align}
Note that the last inequality follows from the basic inequality 
$$
|a+b|^2 \leq \frac{3}{2}|a|^2 + 3|b|^2.
$$
Using Jensen's inequality and Assumption \ref{ass1}, we have
\begin{align}\label{prf-eq5.5}
|(P \ast \mu)(t,x)|^2 \leq \int |P(x-y)|^2 d \mu(t,y)
\leq C_P^2 |x-\bar{x}|^2 + C_P^2 \int |y-\bar{x}|^2 d\mu(t,y).
\end{align}
By \eqref{prf-eq4}-\eqref{prf-eq5.5},  there exists a constant $C_5>0$ depending on $C_P,C_\Psi,C_L$ and $m$, such that
\begin{align}\label{prf-eq6}
\int_s^{s+h}f(\mu(t,x),u(t,x))dt 
    \leq  \frac{3}{2} \frac{C_4}{m}\int_s^{s+h} \int |u(t,x)|^2 d \mu(t,x)dt + C_5 \int_s^{s+h} \int |x-\bar{x}|^2 d \mu(t,x) dt.
\end{align}
On the other hand, by the strict dissipativity we obtain 
\begin{align}
\int_s^{s+h}f(\mu(t,x),u(t,x))dt \geq&~ C_D \int_s^{s+h} \int |x-\bar{x}|^2  + | u(t,x)|^2 d\mu(t,x) dt \nonumber\\
\geq&~ C_D \int_s^{s+h} \int| u(t,x)|^2 d\mu(t,x) dt. \label{prf-eq7}
\end{align}
By equation \eqref{prf-eq6}-\eqref{prf-eq7} we conclude that 
\begin{align*}
\int_s^{s+h} \int | u(t,x)|^2 d\mu(t,x) dt  \leq  \frac{3}{2} \frac{C_4}{m C_D}\int_s^{s+h} \int |u(t,x)|^2 d \mu(t,x)dt 
+ \frac{C_5}{C_D} \int_s^{s+h} \int |x-\bar{x}|^2 d \mu(t,x) dt.
\end{align*}
Set 
\begin{equation}\label{m}
    m = \max\left\{2,  \frac{2C_4}{C_D}\right\},
\end{equation}
and hence, 
$
\frac{3}{2} \frac{C_4}{m C_D} \leq \frac{3}{4}. 
$
Therefore, 
\begin{align*}
\int_s^{s+h} \int | u(t,x)|^2 d\mu(t,x) dt  \leq&~ \frac{3}{4} \int_s^{s+h} \int |u(t,x)|^2 d \mu(t,x)dt + \frac{C_5}{C_D} \int_s^{s+h} \int |x-\bar{x}|^2 d \mu(t,x) dt.
\end{align*}
Since $m$ is given, we know that the constant $C_5>0$  depends only on $C_P,C_\Psi,$ and $C_L$, respectively. This holds for any $h$ satisfying $s+mh\leq T$. By Lebesgue's differentiation theorem we obtain
$
\int | u(s,x)|^2 d\mu(s,x) \leq \frac{4C_5}{C_D}  \int | x-\bar{x}|^2 d\mu(s,x) \mbox{ for a.e. } t\in(0,T).$
Combining this estimate with the results of Theorem \ref{theorem-m-s},  the proof is completed for  
$
C_3 = \frac{4C_2C_5}{C_D}.
$
\end{proof}

\begin{remark}
By Theorem \ref{theorem-m-s} and Theorem \ref{thm14-c}, the cost in \eqref{mec} also decreases exponentially since
\begin{align*}
\int_{0}^T f(\mu(t,x),u(t,x)) dt \leq &~C_L
\int_{0}^T \int |x-\bar{x}|^2 d\mu(t,x) dt + C_{\Psi} \int_{0}^T \int |u(t,x)|^2 d\mu(t,x) dt \\[2mm]
\leq &~(C_L C_2 + C_{\Psi} C_3) e^{-\alpha t} \int |x-\bar{x}|^2 d\mu_0(x).
\end{align*}
\end{remark}

\begin{remark}
    In the proof, we adapt the technique in \cite{EGPZ} by considering a new feedback control and introducing an adaptive parameter $m$ in \eqref{m}. If the cost function in equation \eqref{mec} of quadratic form,
    $$
    f(\mu(t,x),u(t,x))=\int|x-\bar{x}|^2 d\mu(t,x) + \int |u(t,x)|^2 d\mu(t,x),
    $$
    then, we have $C_\Psi = 1$, $C_L=1$ and $C_D=1$. It follows that $C_4=1$ and $m=2$.
\end{remark}

\begin{remark}
The exponential turnpike property for the optimal control problem of the $N$-particles system \eqref{particle-pb} can also be proved by considering the feedback control
$$
\tilde{u}_N(t,\tilde{x}_i(t)) = 
\left\{
\begin{array}{ll}
    u_N(t,\tilde{x}_i(t)), &  \qquad t\in[0,s) \\[2mm]
    \dfrac{1}{m}u_N(t_1,\tilde{x}_i(t))-\dfrac{m-1}{m}\dfrac{1}{N}\sum_{j=1}^NP(\tilde{x}_i(t)-\tilde{x}_j(t)) & \qquad t\in[s,s+mh] \\[4mm]
    u_N(t_2,\tilde{x}_i(t)), &  \qquad t\in(s+mh,T],
\end{array}
\right.
$$
where $t_1$ and $t_2$ are taken as those in the proof of Theorem \ref{thm14-c}.
\end{remark}

\section*{\normalsize{Acknowledgments}} 
The first author thanks the Deutsche Forschungsgemeinschaft (DFG, German Research Foundation) for the financial support through 442047500/SFB1481 within the projects B04 (Sparsity fördernde Muster in kinetischen Hierarchien), B05 (Sparsifizierung zeitabhängiger Netzwerkflußprobleme mittels diskreter Optimierung) and B06 (Kinetische Theorie trifft algebraische Systemtheorie) and though SPP 2298 Theoretical Foundations of Deep Learning  within the Project(s) HE5386/23-1, Meanfield Theorie zur Analysis von Deep Learning Methoden (462234017). The second author is funded by Alexander von Humboldt Foundation (Humboldt Research Fellowship Programme for Postdocs).

\begin{appendices}
\section{Proof of Lemma \ref{lemma-W2-stability} }\label{sectionA}
We follow the idea in \cite{FS,CCR} to prove the estimate in the Wasserstein distance $\mathcal{W}_2$ of Lemma \ref{lemma-W2-stability}: 
\par 
Let $\mathcal{T}^{\mu}_t$ be the flow map associated to the system
$$
\frac{dx(t)}{dt} = (P \ast \mu)(x(t)) + u(t,x(t)) = \int P(x(t)-y)d\mu(t,y) + u(t,x(t)).
$$
We know that $\mu(t) = \mathcal{T}^{\mu}_t \sharp \mu_0$ with $\mathcal{T}^{\mu}_t \sharp$ denotes the push-forward of $\mu_0$. Then, we have 
\begin{align}
\mathcal{W}_2(\mu(t),\nu(t)) = &~\mathcal{W}_2(\mathcal{T}^{\mu}_t \sharp \mu_0,\mathcal{T}^{\nu}_t \sharp \nu_0) \nonumber \\[2mm]
\leq &~\mathcal{W}_2(\mathcal{T}^{\mu}_t \sharp \mu_0,\mathcal{T}^{\mu}_t \sharp \nu_0) + \mathcal{W}_2(\mathcal{T}^{\mu}_t \sharp \nu_0,\mathcal{T}^{\nu}_t \sharp \nu_0). \label{Append-0}
\end{align}
For the first term, we have the following result. 
\begin{lemma}\label{lemma-a1}
% Let $E_1$ and $E_2$ be two bounded Borel measurable functions. Then, for every $\mu \in P_2(\mathbb{R}^d)$ one has
% $$
% \mathcal{W}_2(E_1 \sharp \mu , ~E_2 \sharp \mu) \leq \|E_1 - E_2\|_{\infty}.
% $$
% If in addition $E_1$ is locally Lipschitz continuous, and $\mu,\nu \in P_2(\mathbb{R}^d)$ are both compactly supported on a ball $B_r$, then
Assume that $P$ satisfies the Lipschitz condition \eqref{Lip-P} and $u(t,x)\in \mathcal{F}$. Then, it holds that
$$
\mathcal{W}_2(\mathcal{T}^{\mu}_t \sharp \mu_0 , ~\mathcal{T}^{\mu}_t \sharp \nu_0) \leq e^{(C_P+C_B)t}~ \mathcal{W}_2(\mu_0,\nu_0).
$$
% where $L_r$ is the Lipschitz constant of $E_1$ on $B_r$.
\end{lemma}
\begin{proof}

Set $\kappa$ to be an optimal transportation between $\mu_0$ and $\nu_0$. One can check that the measure
$\gamma = (\mathcal{T}^{\mu}_t \times \mathcal{T}^{\mu}_t)\sharp \kappa$ has marginals $\mathcal{T}^{\mu}_t \sharp \mu_0$ and $\mathcal{T}^{\mu}_t \sharp \nu_0$. 
Then we have
\begin{align}
\mathcal{W}_2(\mathcal{T}^{\mu}_t \sharp \mu_0 , ~\mathcal{T}^{\mu}_t \sharp \nu_0) \leq&~ \left(\int_{\mathbb{R}^d\times \mathbb{R}^d} |x_0 - y_0|^2 d\gamma(x_0, y_0) \right)^{1/2} \nonumber \\[2mm]
=&~ \left(\int_{\mathbb{R}^d\times \mathbb{R}^d} |\mathcal{T}^{\mu}_t(x_0) - \mathcal{T}^{\mu}_t(y_0)|^2  d\kappa(x_0, y_0) \right)^{1/2}. \label{Append-1}
\end{align}
Denote $x(t)=\mathcal{T}^{\mu}_t(x_0)$ and $y(t)=\mathcal{T}^{\mu}_t(y_0)$. We have 
\begin{align*}
    |x(t)-y(t)| 
    \leq&~ |x_0-y_0| + \int_0^t |(P \ast \mu)(x(s)) - (P \ast \mu)(y(s))| + |u(s,x(s))-u(s,y(s))| ds \\[2mm]
    \leq&~ |x_0-y_0| + C_P \int_0^t | x(s)- y(s)| ds + C_B \int_0^t | x(s)- y(s)| ds.
\end{align*}
By Gronwall's inequality, we have 
$$
|x(t)-y(t)| \leq e^{(C_P+C_B)t}|x_0-y_0|.
$$
Substituting this into \eqref{Append-1}, we have
$$
\mathcal{W}_2(\mathcal{T}^{\mu}_t \sharp \mu_0 , ~\mathcal{T}^{\mu}_t \sharp \nu_0) \leq e^{(C_P+C_B)t} \left(\int_{\mathbb{R}^d\times \mathbb{R}^d} |x_0 - y_0|^2 d\kappa(x_0, y_0) \right)^{1/2} = e^{(C_P+C_B)t}~ \mathcal{W}_2(\mu_0,\nu_0).
$$
% \\[2mm]
% \leq &~ L_r \left(\int_{\mathbb{R}^d\times \mathbb{R}^d} |x - y|^2 \kappa(x, y) dx dy \right)^{1/2}\\[2mm]
% = &~ L_r ~\mathcal{W}_2(\mu,\nu).
% \end{align*}
\end{proof}

% Next we turn to estimate $\|\mathcal{T}^{\mu}_t - \mathcal{T}^{\nu}_t\|_{\infty}$ and the Lipschitz constant of $\mathcal{T}^{\mu}_t$.
For the second term in \eqref{Append-0}, we have the following Lemma.
\begin{lemma}
Let $\mathcal{T}^{\mu}_t$ and $\mathcal{T}^{\nu}_t$ be two flow maps  associated to $\mu(t)$ and $\nu(t)$. Suppose the initial data $\nu_0 \in P_2(\mathbb{R}^d)$. Then,
$$
\mathcal{W}_2(\mathcal{T}^{\mu}_t \sharp \nu_0,\mathcal{T}^{\nu}_t \sharp \nu_0) \leq \|\mathcal{T}^{\mu}_t - \mathcal{T}^{\nu}_t\|_{\infty}.
$$
\end{lemma}

\begin{proof}
The proof is similar to that in Lemma 3.11 in \cite{CCR}. 
Consider a transportation plan defined by $\pi:= (\mathcal{T}^{\mu}_t \times \mathcal{T}^{\nu}_t)\sharp \nu_0$. One can
check that this measure has marginals $\mathcal{T}^{\mu}_t\sharp \nu_0$ and $\mathcal{T}^{\nu}_t\sharp \nu_0$. Then, due to the definition of Wasserstein metric, we have
\begin{align*}
\mathcal{W}_2(\mathcal{T}^{\mu}_t \sharp \nu_0, \mathcal{T}^{\nu}_t \sharp \nu_0) 
\leq&~ \left(\int_{\mathbb{R}^d\times \mathbb{R}^d} |x_0 - y_0|^2 \pi(x_0, y_0) dx_0 dy_0 \right)^{1/2} \\[2mm]
=&~ \left( \int_{\mathbb{R}^d} |\mathcal{T}^{\mu}_t(x_0) - \mathcal{T}^{\nu}_t(x_0)|^2 d\nu_0(x_0)\right)^{1/2}\\[2mm]
\leq &~ \|\mathcal{T}^{\mu}_t - \mathcal{T}^{\nu}_t\|_{\infty}.
\end{align*}
\end{proof}

Thanks to this, it suffices to estimate $\|\mathcal{T}^{\mu}_t - \mathcal{T}^{\nu}_t\|_{\infty}$. To this end, we state
\begin{lemma}\label{lemma-a3}
Under the assumptions in Lemma \ref{lemma-a1}, it holds that
$$
\|\mathcal{T}^{\mu}_t - \mathcal{T}^{\nu}_t\|_{\infty} \leq C_P \int_0^t e^{(C_P+C_B)(t-s)}~ \mathcal{W}_2(\mu(s),\nu(s)) ds.
$$
\end{lemma}

\begin{proof}
Denote $x^{\mu}(t) = \mathcal{T}^{\mu}_t(x_0)$ and $x^{\nu}(t) = \mathcal{T}^{\nu}_t(x_0)$. We compute 
\begin{align}\label{App-01}
    |x^{\mu}(t)-x^{\nu}(t)| \leq \int_0^t |(P\ast \mu)(x^{\mu}(s)) - (P\ast \nu)(x^{\nu}(s))|ds + \int_0^t |u(s,x^{\mu}(s)) - u(s,x^{\nu}(s))|ds.
\end{align}  
% From the basic inequality $(a+b)^2\leq 2(a^2+b^2)$ and Jensen's inequality, we have
% \begin{align*}
%     |x^{\mu}(t)-x^{\nu}(t)|^2 \leq 2 \int_0^t \left|(P\ast \mu)(x^{\mu}(s)) - (P\ast \nu)(x^{\nu}(s)) \right|^2 ds + 
%     2\int_0^t \left|u(s,x^{\mu}(s)) - u(s,x^{\nu}(s))\right|^2 ds.
% \end{align*}
For the first term on the right hand side, we compute
\begin{align}
     &~\int_0^t \left|(P\ast \mu)(x^{\mu}(s)) - (P\ast \nu)(x^{\nu}(s)) \right| ds \nonumber \\[2mm]
    \leq&~ \int_0^t |(P\ast \mu)(x^{\mu}(s)) - (P\ast \mu)(x^{\nu}(s))| + |(P\ast \mu)(x^{\nu}(s)) - (P\ast \nu)(x^{\nu}(s))| ds \nonumber \\[2mm]
    \leq&~ C_P\int_0^t |x^{\mu}(s) - x^{\nu}(s)|ds  + \int_0^t \|(P\ast \mu)(s,\cdot) - (P\ast \nu)(s,\cdot) \|_{\infty} ds.\label{App-02}
\end{align}
Moreover, using the fact that $u\in\mathcal{F}$, it follows from \eqref{App-01}-\eqref{App-02} that 
\begin{align*}
|x^{\mu}(t)-x^{\nu}(t)|  \leq&~ \int_0^t (C_P+C_B)|x^{\mu}(s) - x^{\nu}(s)|ds  + \int_0^t \|(P\ast \mu)(s,\cdot) - (P\ast \nu)(s,\cdot) \|_{\infty} ds.   
\end{align*}
By Gronwall's inequality, we have
\begin{align*}
|x^{\mu}(t)-x^{\nu}(t)|  \leq&~ \int_0^t e^{(C_P+C_B)(t-s)}~ \|(P\ast \mu)(s,\cdot) - (P\ast \nu)(s,\cdot) \|_{\infty} ds.   
\end{align*}

Denote $\theta(y,z;t)$ the optimal transportation between $\mu$ and $\nu$. Clearly, $\theta(y,z;t)$ has marginals $\mu(t,y)$ and $\nu(t,z)$. Thus we compute 
\begin{align*}
    (P\ast \mu - P\ast \nu)(t,x) =&~
    \int_{\mathbb{R}^d} P(x-y)d\mu(t,y) - \int_{\mathbb{R}^d} P(x-z)d\nu(t,z)\\[2mm]
    =&~\int_{\mathbb{R}^{2d}} [P(x-y) - P(x-z)] d\theta(y,z;t).
\end{align*}
It follows from Jensen's inequality that
\begin{align*}
    |(P\ast \mu - P\ast \nu)(t,x)| 
    \leq &~\left(\int_{\mathbb{R}^{2d}} |P(x-y) - P(x-z)|^2 d\theta(y,z;t)\right)^{1/2}\\[2mm]
    \leq &~C_P \left(\int_{\mathbb{R}^{2d}} |y-z|^2 d\theta(y,z;t)\right)^{1/2} = C_P \mathcal{W}_2(\mu(t),\nu(t)).
\end{align*}
Note that it holds for arbitrary $x\in \mathbb{R}^d$. Thus we know that 
\begin{align*}
|x^{\mu}(t)-x^{\nu}(t)|  \leq&~ C_P \int_0^t e^{(C_P+C_B)(t-s)}~ \mathcal{W}_2(\mu(s),\nu(s)) ds.   
\end{align*}
\end{proof}

Combining Lemma \ref{lemma-a1}-\ref{lemma-a3} with the inequality \eqref{Append-0}, we have 
\begin{align*}
\mathcal{W}_2(\mu(t),\nu(t)) 
\leq &~\mathcal{W}_2(\mathcal{T}^{\mu}_t \sharp \mu_0,\mathcal{T}^{\mu}_t \sharp \nu_0) + \mathcal{W}_2(\mathcal{T}^{\mu}_t \sharp \nu_0,\mathcal{T}^{\nu}_t \sharp \nu_0)\\[2mm]
\leq &~e^{(C_P+C_B)t}~\mathcal{W}_2(\mu_0,\nu_0) + C_P \int_0^t e^{(C_P+C_B)(t-s)} ~\mathcal{W}_2(\mu(s),\nu(s)) ds.
\end{align*}
Then we have 
\begin{align*}
e^{-(C_P+C_B)t}~\mathcal{W}_2(\mu(t),\nu(t)) 
\leq &~\mathcal{W}_2(\mu_0,\nu_0) + C_P \int_0^t e^{-(C_P+C_B)s} ~\mathcal{W}_2(\mu(s),\nu(s)) ds.
\end{align*}
Again, by Gronwall's inequality, we obtain
$$
e^{-(C_P+C_B)t}~\mathcal{W}_2(\mu(t),\nu(t)) \leq e^{C_P t}~ \mathcal{W}_2(\mu_0,\nu_0),\quad t\in [0,T]. 
$$
This completes the proof of the stability with respect to the $\mathcal{W}_2$ distance.
\end{appendices}

\bibliographystyle{abbrv}
\bibliography{ref}

\end{document}